\documentclass[11pt]{amsart}

\usepackage{amsfonts}
\usepackage{amssymb}
\usepackage{amsmath}
\usepackage{amsthm}
\usepackage{stmaryrd} 
\usepackage{wasysym} 
\usepackage[colorlinks,allcolors=magenta]{hyperref}
\usepackage{tikz}

\usepackage[all]{xy}
\usepackage{enumitem}

\setenumerate{topsep=0pt,labelwidth=0pt,align=left,labelindent=0pt,labelsep=0in,labelwidth=0.275in,leftmargin=0.275in,label=\rm(\alph*),parsep=0pt,itemsep=1pt}

\renewcommand{\phi}{\varphi}

\newcommand{\restr}{\upharpoonright}

\newcommand{\N}{\mathbb{N}}
\newcommand{\Q}{\mathbb{Q}}

\renewcommand{\P}{\mathbb{P}}

\newcommand{\LM}{\mathcal{M}}

\newcommand{\LP}{\mathcal{P}}

\newcommand{\LA}{\mathcal{A}}

\renewcommand{\b}{\mathfrak{b}}
\renewcommand{\d}{\mathfrak{d}}
\renewcommand{\a}{\mathfrak{a}}
\renewcommand{\c}{\mathfrak{c}}

\newtheorem{thm}{Theorem}[section]

\newtheorem{lemma}[thm]{Lemma}
\newtheorem{cor}[thm]{Corollary}

\theoremstyle{definition}

\theoremstyle{remark}
\newtheorem*{remark}{Remark}

\newcommand{\supp}{\mathrm{supp}}

\newcommand{\concat}{^\smallfrown}

\newcommand{\CH}{\mathsf{CH}}
\newcommand{\MA}{\mathsf{MA}}

\newcommand{\bb}{\mathrm{bb}}

\renewcommand{\P}{\mathbb{P}}

\newcommand{\FIN}{\mathrm{FIN}}

\newcommand{\non}{\mathrm{non}}

\title[Mad families of subspaces and the smallest nonmeager set]{Mad families of vector subspaces and the smallest nonmeager set of reals}
\date{September 17, 2019}
\author{Iian B. Smythe}
\address{Department of Mathematics, Rutgers University -- New Brunswick, 110 Frelinghuysen Road, Piscataway, NJ, 08854}
\urladdr{www.iiansmythe.com}
\email{i.smythe@rutgers.edu}

\thanks{The author would like to thank J\"org Brendle and Michael Hru\v{s}\'ak for helpful conversations that have contributed to this work. He would also like to thank the Centro de Ciencias Matem\'{a}ticas at the Universidad Nacional Aut\'{o}noma de M\'{e}xico in Morelia and Casa Matem\'{a}tica Oaxaca for providing the time and space to have these conversations and contemplate the uncountable. Lastly, he must thank his friends Jeffrey Bergfalk and Sena Aydin for making him feel at home in Morelia, and the Department of Mathematics at Rutgers University for funding his travels to Mexico.}

\begin{document}

\begin{abstract}
	We show that a parametrized $\diamondsuit$ principle, corresponding to the uniformity of the meager ideal, implies that the minimum cardinality of an infinite maximal almost disjoint family of block subspaces of a countable vector space is $\aleph_1$. Consequently, this cardinal invariant is $\aleph_1$ in the Miller model. This verifies a conjecture of the author from \cite{Smythe_mad_vec}.
\end{abstract}

\maketitle

\section{Introduction}

The setting of this note is a countably infinite-dimensional vector space $E$ over a countable field $F$, with a distinguished basis $(e_n)_{n\in\N}$. For example, we may take $E=\bigoplus_{n\in\N} F$ and $e_n$ to be the $n$th unit coordinate vector. If $X$ is a set or sequence of vectors in $E$, then $\langle X\rangle$ will denote its linear span. We also fix a well-ordering of $E$ that will be used implicitly.

We will focus on \emph{block subspaces} of $E$, that is, those linear subspaces which have a basis $(x_n)_{n\in\N}$, called a \emph{block sequence}, such that for all $n$,
\[
	\max(\supp(x_n))<\min(\supp(x_{n+1})),
\]
where the \emph{support} of a nonzero vector $v$ is the set $\supp(v)$ of those $n$ for which the coefficient of $e_n$ in the basis expansion of $v$ is nonzero. We will abbreviate $\max(\supp(v))<\min(\supp(w))$ by writing $v<w$. It is easy to see that every infinite-dimensional subspace contains a block subspace.

Two subspaces $X$ and $Y$ of $E$ are \emph{almost disjoint} if their intersection $X\cap Y$ is finite-dimensional. A collection $\LA$ of infinite-dimensional subspaces which are pairwise almost disjoint is called an \emph{almost disjoint family of subspaces}, and it is \emph{maximal almost disjoint}, or \emph{mad}, if it is not properly contained in any other such family. Note that an almost disjoint family consisting of block subspaces is maximal if and only if it is maximal amongst almost disjoint families of block subspaces.

Our focus here is the following cardinal invariant:
\[
	\a_{\mathrm{vec},F}=\min\{|\LA|:\LA \text{ is an infinite mad family of block subspaces}\}.
\]
The basic properties of mad families of block subspaces and $\a_{\mathrm{vec},F}$ can be found in \cite{Smythe_mad_vec}. In particular, $\a_{\mathrm{vec},F}$ is uncountable (Proposition 2.5 in \cite{Smythe_mad_vec}).

When $|F|=2$, our setting corresponds exactly to the study of combinatorial subspaces in $\FIN$, the set of finite nonempty subsets of $\N$. The resulting cardinal invariant, $\a_{\FIN}$, was investigated by Brendle and Garc\'{i}a \'{A}vila \cite{MR3685044}, who proved the following lower bound, which the author observed could be extended to $\a_{\mathrm{vec},F}$ for general $F$:


\begin{thm}[Corollary 2.11 in \cite{Smythe_mad_vec}]
	$\non(\LM)\leq\a_{\mathrm{vec},F}$.
\end{thm}

Here, $\non(\LM)$ is the \emph{uniformity of the meager ideal} $\LM$, that is, the smallest cardinality of a nonmeager subset of the reals, or equivalently, any perfect Polish space. Consequently, in any model of set theory in which $\non(\LM)$ is ``large'', i.e.,~at least $\aleph_2$, so is $\a_{\mathrm{vec},F}$. Such models include those satisfying Martin's Axiom $\MA(\kappa)$ for $\kappa\geq\aleph_1$, as well as the random, Hechler, Laver, and Mathias models (cf.~\cite{MR2768685}).

When, then, is $\a_{\mathrm{vec},F}$ ``small'', i.e.,~equal to $\aleph_1$? It is shown in \cite{Smythe_mad_vec} that $\a_{\mathrm{vec},F}$ is small in the Cohen and (iterated or side-by-side) Sacks models. It was conjectured in \cite{Smythe_mad_vec} that this is also the case in the Miller model (described at the end of \S\ref{sec:cards_diamonds} below), and moreover, that a certain \emph{parametrized $\diamondsuit$ principle}, in the sense of Moore, Hru\v{s}\'{a}k, and D\v{z}amonja \cite{MR2048518}, suffices. It is this conjecture that we verify in Theorem \ref{thm:diamond_a_vec} below.\footnote{The conjecture in \cite{Smythe_mad_vec} states ``$\diamondsuit(\N^\N,=^\infty)$ implies $\a_{\mathrm{vec},F}=\aleph_1$''; this is, possibly, slightly stronger than what we prove here, cf.~Lemma \ref{lem:inf_equal_tukey}.} As a consequence of the main result from \cite{MR2048518} (Theorem \ref{thm:MHD} below), in essentiall all ``canonical'' models of set theory, i.e., those obtained by $\omega_2$-length countable support iterations of definable proper forcings, $\a_{\mathrm{vec},F}=\non(\LM)$.

There are a few important precursors to this work that deserve mention. The first is Hru\v{s}\'{a}k's proof \cite{MR1815092} that the principle $\diamondsuit_\d$ implies that $\a$, the minimum cardinality of an infinite mad family of subsets of $\N$, is $\aleph_1$, and later \cite{MR2048518}  that $\diamondsuit(\N^\N,\not>^*)$, also known as $\diamondsuit(\b)$, suffices. Since $\b\leq\a$, this determines the value of $\a$ in all of the aforementioned canonical models. 

In the sphere of generalized mad families, to which this work belongs, the most relevant precursors are results on the cardinal invariants $\a_e$, $\a_p$, and $\a_g$, the minimum cardinalities of a maximal eventually different family of functions on $\N$, a maximal almost disjoint family of permutations of $\N$, and a maximal cofinitary group of permutations of $\N$, respectively. A characterization of $\non(\LM)$ (Theorem \ref{thm:IOE_nonM} below) implies that $\non(\LM)\leq\a_e$, and moreover, $\diamondsuit(\N^\N,=^\infty)$ implies $\a_e=\aleph_1$ \cite{MR2048518}. Likewise, Brendle, Spinas and Zhang \cite{MR1783921} showed that $\non(\LM)\leq\a_p,\a_g$, and Kastermans and Zhang \cite{MR2258626} showed that $\diamondsuit(\N^\N,=^\infty)$ implies $\a_p=\a_g=\aleph_1$.

This note is arranged as follows: \S\ref{sec:cards_diamonds} reviews the theory of cardinal invariants and parametrized $\diamondsuit$ principles, and isolates the principle $\diamondsuit(\bb^{\infty}(E^2),=^\infty)$ which we will use. The statement and proof of our main result, Theorem \ref{thm:diamond_a_vec}, is given in \S\ref{sec:proof}.

\section{Cardinal invariants and parametrized $\diamondsuit$ principles}\label{sec:cards_diamonds}

Following the terminology in \cite{MR2048518}, an \emph{invariant} is a triple $(A,B,R)$ where $A$ and $B$ are sets and $R\subseteq A\times B$ is a relation which satisfies:
\begin{enumerate}[label=\textup{(\roman*)}]
	\item for every $a\in A$, there is a $b\in B$ such that $aRb$, and
	\item for every $b\in B$, there is an $a\in A$ such that $a\!\not\!\! Rb$.	
\end{enumerate}
$(A,B,R)$ is \emph{Borel} if $A$ and $B$ are Borel subsets of Polish spaces $X$ and $Y$, and $R$ is Borel in $X\times Y$.

The standard reference for the theory of such invariants is \cite{MR2768685}, though we also point to the recent survey \cite{Greenberg_etal_card_invs}, which treats them in the context of computability theory, where they are called \emph{Weihrauch problems}.

The \emph{evaluation} of an invariant $(A,B,R)$ is
\[
	\langle A,B,R\rangle = \min\{|X|:X\subseteq B\text{ and }\forall a\in A\exists b\in B(aRb)\}.
\]
When $A=B$ we will just write $(A,R)$ and $\langle A,R\rangle$, respectively. 

For example, $\langle \N^\N,\not>^*\rangle=\b$, where $f>^*g$ means $\forall^\infty n (f(n)>g(n))$.\footnote{$\forall^\infty$ and $\exists^\infty$ are abbreviations for ``for all but finitely many'' and ``there exists infinitely many'', respectively.} If by $f=^*g$ we mean $\exists^\infty n(f(n)=g(n))$, then we have the following characterization of $\langle\N^\N,=^\infty\rangle$, due to Bartoszy\'{n}ski:

\begin{thm}[cf.~Theorem 2.4 in \cite{MR917147}]\label{thm:IOE_nonM}
	$\langle\N^\N,=^\infty\rangle = \non(\LM)$.
\end{thm}

A streamlined proof of Theorem \ref{thm:IOE_nonM} can be found in \S4 of \cite{Greenberg_etal_card_invs}.

Given a Borel invariant $(A,B,R)$, the corresponding \emph{parameterized $\diamondsuit$ principle} \cite{MR2048518} is the statement:
\begin{quote}
	$\diamondsuit(A,B,R)$: For every Borel $F:2^{<\omega_1}\to A$, there is a $g:\omega_1\to B$ such that for every $f:\omega_1\to 2$, the set $\{\alpha<\omega_1:F(f\restr\alpha)\,R\,g(\alpha)\}$ is stationary.
\end{quote}
Here, $F:2^{<\omega_1}\to A$ is \emph{Borel} if its restriction to $2^\delta$ is Borel for every $\delta<\omega_1$. If $g$ is as above, we say that $g$ is a \emph{$\diamondsuit(A,B,R)$-sequence for $F$}, and if $F(f\restr\alpha)\,R\,g(\alpha)$, then $g$ \emph{guesses $f$ at $\alpha$}.

%
%

The principle $\diamondsuit(A,B,R)$ is a weakening of the classical $\diamondsuit$ principle, and has a similar relationship to $\langle A,B,R\rangle$ as $\diamondsuit$ has to $\c$. In particular, $\diamondsuit(A,B,R)$ implies $\langle A,B,R\rangle\leq\aleph_1$ (Proposition 2.5 in \cite{MR2048518}). The key result about $\diamondsuit(A,B,R)$ is as follows:

\begin{thm}[Theorem 6.6 in \cite{MR2048518}]\label{thm:MHD}
	Let $(A,B,R)$ be a Borel invariant. Suppose that $(\P_\alpha:\alpha<\omega_2)$ is a sequence of Borel partial orders\footnote{In an iteration, a Borel partial order is really a sequence of codes, interpreted in successive stages of the iteration.} such that for each $\alpha<\omega_2$, $\P_\alpha$ is forcing equivalent to $\LP(\{0,1\})^+\times\P_\alpha$.\footnote{The non-triviality assumption that $\P_\alpha$ is equivalent to $\LP(\{0,1\})^+\times\Q_\alpha$ is very mild; in nearly all of classical forcing notions, this can be accomplished by coding a condition in $\LP(\{0,1\})^+$ into the ``first coordinate'' of the generic object.} Let $\P_{\omega_2}$ be the countable support iteration of this sequence. If $\P_{\omega_2}$ is proper, then $\P_{\omega_2}$ forces $\diamondsuit(A,B,R)$ if and only if $\P_{\omega_2}$ forces $\langle A,B,R\rangle\leq\aleph_1$.
\end{thm}

At the risk of overstatement, we will call models obtained by forcing iterations satisfying the hypotheses of Theorem \ref{thm:MHD}, \emph{canonical models}.

Returning to our vector space setting, let $\bb^{\infty}(E^2)$ be the space of all \emph{$2$-block sequences} in $E$, that is, sequences $(x_n^0,x_n^1)_{n\in\N}$ of pairs of nonzero vectors in $E$ such that for all $n$,
\[	
	x_n^0<x_n^1<x_{n+1}^0<x_{n+1}^1.
\]
Note that both $\bb^\infty(E)$ and $\bb^{\infty}(E^2)$ are homeomorphic to $\N^\N$ (see Theorem I.7.7 in \cite{MR1321597}).

We will use the following variation on $\diamondsuit(\N^\N,=^\infty)$:
\begin{quote}
	$\diamondsuit(\bb^{\infty}(E^2),=^\infty)$: For every Borel $F:2^{<\omega_1}\to\bb^{\infty}(E^2)$, there is a $g:\omega_1\to\bb^{\infty}(E^2)$ such that for every $f:\omega_1\to 2$, the set $\{\delta<\omega_1:F(f\restr\delta)=^\infty g(\delta)\}$ is stationary.
\end{quote}

Given (Borel) invariants $(A_0,B_0,R_0)$ and $(A_1,B_1,R_1)$, a \emph{(Borel) Tukey reduction} from $(A_0,B_0,R_0)$ and $(A_1,B_1,R_1)$ is a pair of (Borel) maps $F:A_0\to A_1$ and $G:B_1\to B_0$ such that for all $a\in A_0$, $b\in B_1$, 
\[
	F(a) R_1 b \quad\text{implies}\quad a R_0 G(b).
\]
These are called \emph{morphisms} in \cite{MR2768685} and \cite{Greenberg_etal_card_invs}. We write $(A_0,B_0,R_0)\leq_B(A_1,B_1,R_1)$ if there is a Borel Tukey reduction from $(A_0,B_0,R_0)$ to $(A_1,B_1,R_1)$. The following lemma is immediate from the definitions.

\begin{lemma}\label{lem:tukey_leq}
	If $(A_0,B_0,R_0)\leq_B(A_1,B_1,R_1)$, then
	\begin{enumerate}[label=\textup{(\roman*)}]
		\item $\langle A_0,B_0,R_0\rangle\leq\langle A_1,B_1,R_1\rangle$,
		\item $\diamondsuit(A_1,B_1,R_1)$ implies $\diamondsuit(A_0,B_0,R_0)$.\qed
	\end{enumerate}
\end{lemma}

The natural invariant corresponding to $\non(\LM)$ is $(\LM,\N^\N,\not\ni)$,\footnote{$(\LM,\N^\N,\not\ni)$ is not, technically, Borel, but it has an equivalent Borel presentation.} and so $\diamondsuit(\LM,\N^\N,\not\ni)$ is usually written as $\diamondsuit(\non(\LM))$.

\begin{lemma}\label{lem:inf_equal_tukey}
	\begin{enumerate}[label=\textup{(\alph*)}]
		\item $(\N^\N,=^\infty)\leq_B(\bb^{\infty}(E^2),=^\infty)$.
		\item $(\bb^{\infty}(E^2),=^\infty)\leq_B(\LM,\N^\N,\not\ni)$.
		\item $\langle\bb^{\infty}(E^2),=^\infty\rangle=\non(\LM)$.
	\end{enumerate}
\end{lemma}

\begin{proof}
	(a) It makes no difference if we replace $\N^\N$ with $(\N\setminus\{0\})^{\N}$, so we do so. Consider the map $F:(\N\setminus\{0\})^\N\to\bb^{\infty}(E^2)$ defined by 
	\[
		F(f)(k)=(v,w),
	\]
	where $(v,w)$ is the least pair of nonzero vectors $v<w$ in $E$ with supports above the vectors in $F(f)(k-1)$ and such that $|\supp(v)|=f(k)$, for all $k$. Let $G:\bb^\infty(E)\to(\N\setminus\{0\})^{<\N}$ be given by 	\[
		G((x_n^0,x_n^1)_{n\in\N})(k)=|\supp(x_k^0)|
	\]
	for all $k$. Then, $F,G$ forms a Borel Tukey reduction as desired.
	
	(b) As mentioned above, $\bb^\infty(E^2)$ is homeomorphic to $\N^\N$, so we will instead show that $(\bb^{\infty}(E^2),=^\infty)\leq_B(\LM,\bb^{\infty}(E^2),\not\ni)$, mimicing the usual proof that $(\N^\N,=^\infty)\leq_B(\LM,\N^\N,\not\ni)$.
	
	Let $F:\bb^\infty(E^2)\to\LM$ be defined by
	\[
		F(X)=\{Y=(y_n^0,y_n^1)_{n\in\N}\in\bb^{\infty}(E^2):\forall^\infty n ((y_n^0,y_n^1)\neq (x_n^0,x_n^1))\},
	\]
	where $X=(x_n^0,x_n^1)_{n\in\N}$. It is easy to check that $F(X)$ is a meager $F_\sigma$ set in $\bb^{\infty}(E^2)$. Let $G:\bb^{\infty}(E^2)\to\bb^{\infty}(E^2)$ be the identity map. If $F(X)\not\ni Y$, then $X=^\infty Y$, showing that $F,G$ give a Tukey reduction from $(\bb^{\infty}(E^2),=^\infty)$ to $(\LM,\bb^{\infty}(E^2),\not\ni)$. The maps $F,G$ are clearly Borel.
	
	(c) follows from (a), (b), and Theorem \ref{thm:IOE_nonM}, by Lemma \ref{lem:tukey_leq}.
\end{proof}

In particular, $\diamondsuit(\LM,\N^\N,\not\ni)$ implies $\diamondsuit(\bb^{\infty}(E^2),=^\infty)$, which in turn implies $\diamondsuit(\N^\N,=^\infty)$. In fact, $(\LM,\N,\not\ni)\leq_B(\N^\N,=^\infty)\ast(\N^\N,=^\infty)$ (Proposition 4.14 in \cite{Greenberg_etal_card_invs}),\footnote{Curiously, a result of Zapletal \cite{MR3193422} implies that $(\LM,\N,\not\ni)\not\leq_B(\N^\N,=^\infty)$.} where the latter is the (Borel) sequential composition of the invariant $(\N^\N,=^\infty)$ with itself (cf.~Definition 4.10 in \cite{MR2768685} or \S4.1 of \cite{Greenberg_etal_card_invs}). As a result, in any canonical model, the principles $\diamondsuit(\N^\N,=^\infty)$, $\diamondsuit(\bb^{\infty}(E^2),=^\infty)$, and $\diamondsuit(\non(\LM))$, are equivalent.\footnote{The question of whether apparently different $\diamondsuit$ principles which have the same evaluation are \emph{actually different} remains, stubbornly, open.}

Lastly, we briefly recall the description of the Miller model: Let $\Q$ denote \emph{Miller forcing} \cite{MR763899}, the set of all trees $q\subseteq\N^{<\N}$ such that for each $t\in q$, the corresponding branching set $\{n:t\concat(n)\in q\}$ is infinite (such trees are called \emph{rational} or \emph{superperfect}). We order $\Q$ by containment. By the \emph{Miller model} we mean the result of an $\omega_2$-length countable support iteration of Miller forcing over a model $\CH$. The relevant fact about the Miller model is:

\begin{thm}[cf.~Theorem 7.3.46 in \cite{MR1350295}]\label{thm:nonM_Miller}
	$\non(\LM)=\aleph_1$ in the Miller model.
\end{thm}

Since the Miller model is canonical, Theorems \ref{thm:nonM_Miller} and \ref{thm:MHD} imply:

\begin{cor}\label{cor:diamond_Miller}
	$\diamondsuit(\bb^{\infty}(E^2),=^\infty)$ holds in the Miller model.\qed
\end{cor}

\section{The result}\label{sec:proof}

We can now state and prove our main result.
	
\begin{thm}\label{thm:diamond_a_vec}
	$\diamondsuit(\bb^{\infty}(E^2),=^\infty)$ implies $\a_{\mathrm{vec},F}=\aleph_1$.
\end{thm}

We will need the following lemma, which in the $n=0$ case is just a restatement of Lemma 2.3 in \cite{Smythe_mad_vec}. The general case can be obtained by applying this fact repeatedly.

\begin{lemma}\label{lem:extend}
	Let $A_0,\ldots, A_n$ be block subspaces of $E$. For any $m\in\N$, there is an $M\in\N$ such that for any finite-dimensional subspace $V$ of $E$ with $\supp(x)\subseteq[0,m]$ for all $x\in V$, if $y>M$, then for all $i\leq n$,
	\[
		(V+\langle y\rangle)\cap A_i=V\cap A_i \quad\text{if and only if}\quad y\notin A_i.\qed
	\]
\end{lemma}

\begin{proof}[Proof of Theorem \ref{thm:diamond_a_vec}.]
	To define $F:2^{<\omega_1}\to\bb^\infty(E^2)$, by a standard coding, we will instead define $F$ on pairs of the form $((A_\xi)_{\xi<\delta},B)$ where $\{A_\xi:\xi<\delta\}$ is an almost disjoint family of (possibly finite-dimensional) block subspaces of $E$, $\omega\leq\delta<\omega_1$, and $B$ an infinite-dimensional subspace of $E$ which is almost disjoint from every $A_\xi$. On inputs not of this type, we will take $F$ to be some arbitrary constant value in $\bb^{\infty}(E^2)$. $F$ is clearly Borel.
	
	Fix bijections $e_\delta:\omega\to\delta$ for each $\omega\leq\delta<\omega_1$. Given $((A_\xi)_{\xi<\delta},B)$ as above, begin by letting $F((A_\xi)_{\xi<\delta},B)(0)=(x_0^0,x_0^1)$ to be the least pair of nonzero vectors in $B$ with $x_0^0<x_0^1$. Continuing recursively, suppose we have defined $F((A_\xi)_{\xi<\delta},B)(i)$ for $i<n$. Define $F((A_\xi)_{\xi<\delta},B)(n)=(x_n^0,x_n^1)$ as follows: Let $x_n^0$ be the least $x\in B$ above $x_{n-1}^1$. By Lemma \ref{lem:extend}, choose an $M\geq \max(\supp(x_{n}^0))$ such that for any subspace $V$ of $E$ with supports contained in $[0,\max(\supp(x_{n}^0))]$, if $y\notin\bigcup_{i<n}A_{e_\delta(i)}$ and $y>M$, then
	\[
		(V+\langle y\rangle)\cap A_{e_\delta(i)}=V\cap A_{e_\delta(i)}.
	\]
	Let $x_n^1$ be the least $y\in B$ above $M$ and such that $y\notin\bigcup_{i<n}A_{e_\delta(i)}$. This can be arranged since $B$ is almost disjoint from each $A_{e_\delta(i)}$, for $i<n$.

	
	\begin{remark}
		One might expect that we would define $F((A_\xi)_{\xi<\delta},B)\in\bb^\infty(E)$, choosing the $n$th coordinate in $B$ so that it does not effect the intersection of the subspace we're building with $A_{e_\delta(i)}$ for $i<n$, and that any block sequence infinitely often equal it will suffice for the rest of the argument. However, vectors in such a sequence may fail to be far enough ``above'' previous vectors. It is for this reason that we choose \emph{two} vectors at each stage. The first, $x_n^0$, acts as a firewall, preventing earlier vectors in any element of $\bb^\infty(E^2)$ whose $n$th entry coincides with the pair $(x_n^0,x_n^1)$ from getting too close to the second vector, $x_n^1$, which we will actually use.
	\end{remark}

	Let $g$ be a $\diamondsuit(\bb^{\infty}(E^2),=^\infty)$-sequence for $F$. For each $n\in\N$, denote $g(\delta)(n)$ by $(g(\delta)_n^0,g(\delta)_n^1)$. We define $\LA=\{A_\delta:\delta<\omega_1\}$ recursively. First, let $(A_n)_{n<\omega}$ be any sequence of disjoint, infinite-dimensional block subspaces. Having defined $(A_\xi)_{\xi<\delta}$ for $\omega\leq\delta<\omega_1$, we define $A_\delta=\langle (w_n)_{n\in\N}\rangle$ where $(w_n)_{n\in\N}$ is a block sequence chosen as follows: First, set $w_0=g(\delta)_0^1$. Having defined $w_0,\ldots,w_{n-1}=g(\delta)_{k_{n-1}}^1$, we scan across the sequence $g(\delta)$, starting at index $k_{n-1}+1$, ``attempting'' to set $w_n=g(\delta)_k^1$. We do so when we find the least $k\geq k_{n-1}+1$ such that for all $i<n$,
	\[
		\langle w_0,\ldots,w_{n-1},g(\delta)_k^1\rangle\cap A_{e_\delta(i)}=\langle w_0,\ldots,w_{n-1}\rangle\cap A_{e_\delta(i)},
	\]
	if such a $k$ exists. If such a $k$ does not exist, then $w_m$ remains undefined for all $m\geq n$, i.e., $A_\delta$ is a finite-dimensional space. In any case, $A_\xi\cap A_\delta$ is finite-dimensional for all $\xi<\delta$.
	
	Consider the set  $S=\{\delta<\omega_1:A_\delta\text{ is infinite}\}$. If $B$ is such that $((A_\xi)_{\xi<\delta},B)$ is as described in the first paragraph above, and $g$ guesses $F((A_\xi)_{\xi<\delta},B)$, then $\delta\in S$. Since countable almost disjoint families of block subspaces fail to be maximal, it follows that $S$ is stationary.

	We claim that $\LA=\{A_\delta:\delta\in S\}$ is a mad family of subspaces. Towards a contradiction, suppose that $B$ is an infinite-dimensional subspace almost disjoint from each $A_\delta$, for $\delta\in S$. Then, $g$ guesses $F((A_\xi)_{\xi<\delta},B)$ at some $\delta\geq\omega$, and by the preceding paragraph, $\delta\in S$. Let $k\in\N$ be such that $g(\delta)(k)=F((A_\xi)_{\xi<\delta},B)(k)=(x_k^0,x_k^1)$ and suppose that $n$ is the first value for which we ``attempt'' to set $w_n=g(\delta)_k^1$. In particular, $n\leq k$. By definition of $\bb^\infty(E^2)$, we know that $w_0,\ldots,w_{n-1}$ have supports contained in $[0,\max(\supp(x_k^0))]$, and thus $x_k^1$ has the property that for all $i<k$, and in particular, for $i<n$,
	\[
		\langle w_0,\ldots,w_{n-1},x_k^1\rangle\cap A_{e_\delta(i)}=\langle w_0,\ldots,w_{n-1}\rangle\cap A_{e_\delta(i)},
	\]
	and so we must have set $w_n=x_k^1$. Since this happens for all values of $k$ for which $g(\delta)(k)=F((A_\xi)_{\xi<\delta},B)(k)$, it follows that $A_\delta$ has infinite-dimensional intersection with $B$, contrary to our supposition.
\end{proof}

Combining this with Corollary \ref{cor:diamond_Miller}, we have:

\begin{cor}
	$\a_{\mathrm{vec},F}=\aleph_1$ in the Miller model.	\qed
\end{cor}

As mentioned in the introduction, our Theorem \ref{thm:diamond_a_vec}, together with results from \cite{MR2048518}, \cite{MR1783921}, and \cite{MR2258626}, respectively, implies that $\a_{\mathrm{vec},F}$, $\a_e$, $\a_g$, and $\a_p$ are all equal to $\non(\LM)$ in canonical models, and in particular, to each other. While these cardinals can be separated from $\non(\LM)$ using Shelah's technique \cite{MR2096454} of template iterations, see \cite{MR1928381}, it would be very interesting to find a model in which \emph{some} of $\a_e$, $\a_p$, $\a_g$, and $\a_{\mathrm{vec},F}$ (for some $F$) are unequal. This question was asked for $\a_p$ and $\a_g$ in \cite{MR2258626}, and we believe that it deserve renewed attention.

\bibliography{/Users/iian/Dropbox/Mathematics/math_bib}{}
\bibliographystyle{abbrv}

\end{document}